\documentclass[11pt,reqno]{amsart}
\usepackage{indentfirst, amssymb,amsmath,amsthm, mathrsfs,cite,graphicx,float}
\newtheorem*{theoA}{Theorem A}
\newtheorem*{theoB}{Theorem B}
\newtheorem*{theoC}{Theorem C}
\newtheorem*{theoD}{Theorem D}

\newtheorem{theo}{Theorem}[section]
\newtheorem{lem}{Lemma}[section]

\newtheorem{ques}{Question}[section]

\newtheorem{open problem}{Open problem}[section]
\newcommand{\pa}{\partial}
\newcommand{\ol}{\overline}
\newcommand{\be}{\begin{equation}}
\newcommand{\ee}{\end{equation}}
\newcommand{\bs}{\begin{small}}
\newcommand{\es}{\end{small}}
\newcommand{\beas}{\begin{eqnarray*}}
\newcommand{\eeas}{\end{eqnarray*}}
\newcommand{\bea}{\begin{eqnarray}}
\newcommand{\eea}{\end{eqnarray}}
\renewcommand{\epsilon}{\varepsilon}
\numberwithin{equation}{section}

\begin{document}
\title[Bohr phenomenon in complex Banach spaces]{Bohr phenomenon for certain integral operators and transforms in complex Banach spaces}
\author[V. Allu, R. Biswas, R. Mandal and H. Yanagihara ]{ Vasudevarao Allu, Raju Biswas, Rajib Mandal and Hiroshi Yanagihara}
\date{}
\address{Vasudevarao Allu, Department of Mathematics, Indian Institute of Technology Bhubaneswar, School of Basic Science, Bhubaneswar-752050, Odisha, India.}
\email{avrao@iitbbs.ac.in}
\address{Raju Biswas, Department of Mathematics, Raiganj University, Raiganj, West Bengal-733134, India.}
\email{rajubiswasjanu02@gmail.com}
\address{Rajib Mandal, Department of Mathematics, Raiganj University, Raiganj, West Bengal-733134, India.}
\email{rajibmathresearch@gmail.com}
\address{Hiroshi Yanagihara, Department of Applied Science, Faculty of Engineering, Yamaguchi University, Tokiwadai, Ube, 755-8611, Japan}
 \email{hiroshi@yamaguchi-u.ac.jp}
\let\thefootnote\relax
\begin{abstract}
In this paper, we investigate several Bohr radii associated with the Ces\'aro operator, Bernardi integral operator, $\beta$-Ces\'aro operator, and discrete Fourier transform, all defined on a set of holomorphic mappings from the unit ball of a complex Banach space into the closure of the unit polydisc $\mathbb{D}^n$ within the space $\mathbb{C}^n$.
\end{abstract}
\maketitle
\section{Introduction and Preliminaries}
\noindent Let $f(z)=\sum_{n=0}^{\infty} a_nz^n$ be a holomorphic mapping in the unit disk $\mathbb{D}:=\{z\in\mathbb{C}:|z|<1\}$. The classical Bohr's theorem states that if 
$|f(z)|\leq 1$ in $\Bbb{D}$, then 
\be\label{e12}\sum_{n=0}^{\infty}| a_n| r^n\leq 1\quad\text{for all}\quad |z|=r\leq 1/3.\ee
The number $1/3$ cannot be improved and is known as the Bohr radius (see \cite{B1914}). The inequality (\ref{e12}) is known as the classical Bohr inequality. 
In fact, Bohr's \cite{B1914} original proof of the inequality (\ref{e12}) was limited to the case when $r$ was no greater than $1/6$. Subsequently, Wiener, Schur, and Riesz (see \cite{S1927, T1962}) independently proved that the inequality (\ref{e12}) holds for $r\le 1/3$.
Bohr's and Wiener's proofs can be found in \cite{B1914}. \\[2mm]
\indent In the recent years, there has been a significant increase in research studies concerning Bohr-type theorems in various abstract settings, including refinements, ramifications, and extensions. For a detailed study on the Bohr radius in one complex variable, we refer to \cite{ABM2025,KKP2021, AAH2022,AKP2019,1AH2021,2AH2021,AA2023,BB2004,1KP2018,LP2023,MBG2024,ABM2024, BM2024,PVW2020,K2022, PW2020, AH2022,1AH2022, KP2017} and the references therein.
A natural question arises: ``Is it possible to prove the Bohr-type inequality for certain complex integral operators and integral transforms defined on different function spaces?''
The idea was initially proposed in the context of the unit disk $\mathbb{D}$ for the classical Ces\'aro operator in \cite{KKP2020,KKP2021,KKP2022} and for the Bernardi integral operator in \cite{KS2021}. The Ces\'aro operator is studied in \cite{HL1932} and is defined as follows:
\bea\label{e14} \mathcal{C}f(z):=\sum_{n=0}^\infty \frac{1}{n+1}\left(\sum_{k=0}^n a_k\right) z^n=\int_{0}^1\frac{f(t z)}{1-t z}\; d t\eea
for a holomorphic function $f(z)=\sum_{n=0}^\infty a_nz^n$ in the unit disk $\mathbb{D}$.
Evidently,
\beas \vert\mathcal{C}f(z)\vert=\left\vert\sum_{n=0}^\infty \frac{1}{n+1}\left(\sum_{k=0}^n a_k\right) z^n\right\vert\leq \frac{1}{r}\log\frac{1}{1-r}\quad\text{for each}\quad\vert z\vert=r<1.\eeas 
For the Ces\'aro operator, Kayumov {\it et al.}\cite{KKP2020} have established the following Bohr-type inequality.
\begin{theoA}\cite[Theorem 1, p. 616]{KKP2020} Let $f(z)=\sum_{n=0}^{\infty} a_nz^n$ be holomorphic in $\mathbb{D}$ and $|f(z)|\leq 1$ in $\mathbb{D}$, then 
\beas \mathcal{C}_f(r)=\sum_{n=0}^\infty \frac{1}{n+1}\left(\sum_{k=0}^n |a_k|\right) r^n\leq  \frac{1}{r}\log\frac{1}{1-r}\quad\text{for}\quad\vert z\vert=r\leq R_1,\eeas
where $R_1(\approx0.5335)$ is the positive root of the equation $2x=3(1-x)\log (1/(1-x))$. The number $R_1$ cannot be improved.\end{theoA}
For a holomorphic function $f(z)=\sum_{n=m}^\infty a_n z^n$ in the unit disk $\mathbb{D}$, the Bernardi integral operator (see \cite[p. 11] {MM2000}) is defined as follows:
\bea\label{e16} \mathcal{L}_\beta f(z):=\frac{1+\beta}{z^\beta}\int_0^z f(\eta) \eta^{\beta-1} d \eta=(1+\beta)\sum_{n=m}^\infty \frac{a_n}{\beta+n} z^n, \eea
where $\beta>-m$ and $m\geq 0$ is an integer.
In 2021, Kumar and Sahoo \cite{KS2021} studied the following Bohr-type inequality for the Bernardi integral operator.
\begin{theoB}\cite{KS2021} Let $\beta>-m$. If $f(z)=\sum_{n=m}^\infty a_n z^n$ is holomorphic in $\mathbb{D}$ and $|f(z)|\leq 1$ in $\mathbb{D}$, then 
\beas \sum_{n=m}^\infty \frac{|a_n|}{\beta+n} r^n\leq \frac{r^m}{m+ \beta}\quad\text{for}\quad |z|=r\leq R_2(\beta),\eeas
where $R_2(\beta)$ is the positive root of the equation 
\beas\frac{x^m}{m+\beta}-2\sum_{n=m+1}^\infty \frac{x^n}{n+\beta}=0.\eeas 
The number $R_2(\beta)$ cannot be improved.\end{theoB}
\indent Let $\left\{x_n\right\}_{n=0}^{N-1}$ be a sequence of complex numbers. The discrete Fourier transform (see \cite[Chapter 6]{BN2000}) is defined by 
\beas \mathcal{F}(x_n)=\sum_{n=0}^{N-1} x_n e^{-2\pi i n k/N}. \eeas 
For a holomorphic function $f(z)=\sum_{n=0}^\infty a_n z^n$ in $\mathbb{D}$ with $|f(z)|\leq 1$ in $\mathbb{D}$, we perform the discrete
Fourier transform on the coefficients $a_k$ from $k=0$ to $n$ which gives
\beas \mathcal{F}[f](z)=\sum_{n=0}^\infty \left(\sum_{k=0}^n a_k e^{-2\pi i n k/(n+1)}\right) z^n.\eeas
To obtain the Bohr-type inequality, we denote the majorant series of $\mathcal{F}[f](z)$ as
\bea\label{e15} \mathcal{F}_f(r):=\sum_{n=0}^\infty \left(\sum_{k=0}^n |a_k|\right) r^n, \quad\text{where}\quad |z|=r<1.\eea
For the discrete Fourier transform, Ong and Ng \cite{ON2024} have established the following Bohr-type inequality.
\begin{theoC}\cite{ON2024}
If $f(z)=\sum_{n=0}^\infty a_n z^n$ is holomorphic in $\mathbb{D}$ and $|f(z)|\leq 1$ in $\mathbb{D}$, then 
\beas \mathcal{F}_f(r)\leq \frac{1}{1-r}\quad\text{for}\quad r\leq \frac{1}{3}.\eeas
The constant $1/3$ cannot be improved.
\end{theoC}
 Motivated by \textrm{Theorem A}, Kumar and Sahoo \cite{KS2021} have considered the $\beta$-Ces\'aro operator ($\beta>0$), which is defined as
 \be\label{g6} \mathcal{C}_{\beta}^f(z):=\sum_{n=0}^\infty \frac{1}{n+1}\left(\sum_{k=0}^n\frac{\Gamma(n-k+\beta)}{\Gamma(n-k+1)\Gamma(\beta)}a_k\right) z^n=\int_{0}^1\frac{f(t z)}{(1-t z)^{\beta}} d t=\frac{1}{z}\int_{0}^z\frac{f(t)}{(1-t)^{\beta}} d t\ee
for a holomorphic function $f(z)=\sum_{n=0}^\infty a_nz^n$ in the unit disk $\mathbb{D}$. 
Indeed, Kumar and Sahoo \cite{KS2021} have established the following Bohr-type inequality for the $\beta$-Ces\'aro operator.
\begin{theoD}\cite{KS2021}
Let $0<\beta (\not=1)$. If $f(z)=\sum_{n=0}^\infty a_n z^n$ is holomorphic in $\mathbb{D}$ and $|f(z)|\leq 1$ in $\mathbb{D}$, then 
\beas\sum_{n=0}^\infty \frac{1}{n+1}\left(\sum_{k=0}^n\frac{\Gamma(n-k+\beta)}{\Gamma(n-k+1)\Gamma(\beta)}|a_k|\right) r^n\leq \frac{1}{r}\left(\frac{1-(1-r)^{1-\beta}}{1-\beta}\right)\eeas
for $|z|=r\leq R_3(\beta)$, where $R_3(\beta)$ is the positive root of the equation 
\beas \frac{3\left(1-(1-r)^{1-\beta}\right)}{1-\beta}-\frac{2\left((1-r)^{-\beta}-1\right)}{\beta}=0.\eeas
The number $R_3(\beta)$ cannot be improved.
\end{theoD}
\noindent In particular, if we put $f(t)\equiv 1$, $f(t)\equiv \sum_{n=1}^{\infty}t^n$ and $f(t)\equiv \sum_{n=0}^{\infty}t^n$ in (\ref{g6}), then we obtain, respectively, the following relations:
\bea\label{ceq}&& \sum_{n=0}^\infty \frac{1}{n+1}\left(\frac{\Gamma(n+\beta)}{\Gamma(n+1)\Gamma(\beta)}\right) z^n=\frac{1}{z}\int_{0}^z\frac{1}{(1-t)^{\beta}} d t,\\[2mm]
\label{ceq1}&&\sum_{n=1}^\infty \frac{1}{n+1}\left(\sum_{k=1}^n\frac{\Gamma(n-k+\beta)}{\Gamma(n-k+1)\Gamma(\beta)}\right) z^n=\frac{1}{z}\int_{0}^z\frac{t}{(1-t)^{\beta+1}} d t\\[2mm]\text{and}
\label{ceq2}&&\sum_{n=0}^\infty \frac{1}{n+1}\left(\sum_{k=0}^n\frac{\Gamma(n-k+\beta)}{\Gamma(n-k+1)\Gamma(\beta)}\right) z^n=\frac{1}{z}\int_{0}^z\frac{1}{(1-t)^{\beta+1}} d t.\eea
From (\ref{ceq}) and (\ref{ceq2}), we obtain the following identity
\bea \label{ceq3}\frac{\Gamma(n+\beta+1)}{\Gamma(n+1)\Gamma(\beta+1)}\equiv \sum_{k=0}^n\frac{\Gamma(n-k+\beta)}{\Gamma(n-k+1)\Gamma(\beta)}.\eea
Before proceeding with the discussion, it is essential to introduce the requisite notations. 
For $z=(z_1,z_2,\cdots,z_n)\in\mathbb{C}^n$, let $\Vert z\Vert_\infty=\max\limits_{j=1,2,\cdots,n}|z_j|$ denote the maximum norm.
A ploydisc in $\mathbb{C}^n$ is the product of discs. The open polydisc with center $\alpha=(\alpha_1,\alpha_2,\ldots, \alpha_n)$ and poly-radius $r=(r_1, r_2,\ldots, r_n)$ will be denoted by $\mathbb{D}^n(\alpha; r)$ and defined as  
\beas \mathbb{D}^n(\alpha; r)=\left\{z=(z_1,z_2,\cdots,z_n)\in\mathbb{C}^n: |z_j-\alpha_j|< r_j\quad\text{for}\quad j=1,2,\ldots, n\right\}.\eeas
Let $\alpha=0=(0, 0,\ldots, 0)$ and $r=1=(1,1,\ldots,1)$. Then, $\mathbb{D}^n(0; 1)\equiv \mathbb{D}^n$ denote the unit polydisc in $\mathbb{C}^n$. The set 
\beas \pa_0 \mathbb{D}^n(\alpha; r):=\left\{z=(z_1,z_2,\cdots,z_n)\in\mathbb{C}^n: |z_j-\alpha_j|=r_j\quad\text{for}\quad j=1,2,\ldots, n\right\}\eeas
is contained within the boundary $\pa\mathbb{D}^n(\alpha; r)$ of $\mathbb{D}^n(\alpha; r)$ and it is called the distinguished boundary of $\pa \mathbb{D}^n(\alpha; r)$.\\[2mm]
\indent For a comprehensive overview of Bohr's phenomenon of holomorphic or pluriharmonic mappings with values in higher-dimensional complex Banach spaces using the homogeneous polynomial expansions, see \cite{HH2024, 1HH2024, KPW2025, A2000, LW2007, HHK2009, HHK2025}. Kayumov and Ponnusamy \cite{KP2018} have obtained the Bohr radius for the 
 holomorphic functions with lacunary series in the unit disc $\mathbb{D}$. 
 In their study, Liu and Liu \cite{LL2020} made significant contributions to the field by extending the results of Kayumov and Ponnusamy \cite{KP2018} to holomorphic mappings $f$ with lacunary series from the unit polydisc $\mathbb{D}^n$ into $\mathbb{D}^n$ and from the unit ball $B$ of a complex Banach space into $B$.
Lin {\it et al.} \cite{LLP2023} have studied other types of Bohr's inequality for holomorphic mappings $f$ with lacunary series from the unit polydisc $\mathbb{D}^n$ into $\mathbb{D}^n$.\\[2mm]
\indent Let $X$ and $Y$ be complex Banach spaces. A mapping $P: X\to Y$ is called a homogeneous polynomial of degree $k\in\mathbb{N}$ if there exists a $k$-linear mapping 
$u :X^k\to Y$ such that $P(x)=u(x,x,\ldots,x)$ for every $x\in X$. Note that if $P$ is a homogeneous polynomial of degree $k$, then $P(\mu x)=\mu^k P(x)$ for all $x\in X$ and $\mu\in \mathbb{C}$.\\[2mm]
\indent In this paper, the degree of a homogeneous polynomial is denoted by a subscript. Note that if $P_k$ is a $k$-homogeneous polynomial from $X$ into $Y$, there exists a unique symmetric $k$-linear mapping $u$ such that $P_k(x)=u(x,...,x)$ (see \cite[Chapter 1 and 2]{M1986}). For a domain $\mathcal{D}\subset X$, let $F$ be a holomorphic mapping from $\mathcal{D}$ into $Y$. For $z\in\mathcal{D}$,
 let $D^kF(z)$ denote the $k$-th Fr\'echet derivative of $F$ at $z$. If $\mathcal{D}$ contains the origin, then
 any holomorphic mapping $F : \mathcal{D} \to Y$ can be expanded into the series
 \bea\label{g7} F(z) =\sum_{k=0}^\infty \frac{1}{k!} D^kF(0)(z^k),\eea
 for all $z$ in some neighborhood of the origin (see \cite[Chapter 6, p. 197]{GK2003}). Since $(1/k!)D^kF(0)(z^k)$ is a homogeneous polynomial of degree $k$, we use the notation $P_k(z)=(1/k!)D^kF(0)(z^k)$ all along this paper. Note that if $\mathcal{D}$ is a bounded balanced domain in a complex Banach space $X$ and $F(\mathcal{D})$ is bounded, then (\ref{g7}) converges uniformly on $r\mathcal{D}$ for each $r\in (0,1)$.\\[2mm]
\indent  Let $L(X, \mathbb{C})$ denote the set of continuous linear operators from $X$ into $\mathbb{C}$. For each $x\in X \setminus\{0\}$, let
 \beas T(x)=\left\{l_x \in L(X, \mathbb{C}): l_x(x)=\Vert x\Vert, ~\Vert l_x\Vert =1\right\},\eeas
 where $\Vert l_x\Vert=\sup\{|l_x(w)| : \Vert w\Vert=1\}$.
In view of the Hahn-Banach theorem, $T(x)$ is nonempty. 
The following question naturally arises:
 \begin{ques}\label{Q1} Let $B_X$ be the unit ball of the complex Banach space $X$. Is it possible to establish the Bohr-type inequalities for the Ces\'aro operator, Bernardi integral operator, $\beta$-Ces\'aro operator $(\beta\not=1)$, and discrete Fourier transform defined on a set of holomorphic mappings $F:B_X\to\ol{\mathbb{D}^n}$?\end{ques}
 This paper primarily provides an affirmative answer to Question \ref{Q1}.
\section{Main results}
\noindent The following are key lemmas of this paper and will be used to prove the main results.
\begin{lem}\cite{DP2008}\label{lem} If $f(z)=\sum_{n=0}^\infty a_n z^n$ is holomorphic in $\mathbb{D}$ with $|f(z)|\leq1$ in $\mathbb{D}$, then we have
\beas \frac{\left|f^{(n)}(z)\right|}{n!}\leq \frac{1-|f(z)|^2}{(1-|z|)^{n-1}(1-|z|^2)}\quad\text{and}\quad |a_n|\leq 1-|a_0|^2\quad\text{for}\quad n\geq 1,\;|z|<1.\eeas\end{lem}
\noindent Before proceeding to the main result, it is necessary to establish the majorant series corresponding to the Ces\'aro operator for a holomorphic mapping $F:B_X\to\ol{\mathbb{D}^n}$ with 
\beas F(z)=a+\sum_{s=1}^\infty P_s(z),\quad z\in B_X,\eeas
where $P_s(z)=(1/s!) D^s F(0)(z^s)$ and $F(z)=(F_1(z),F_2(z),\ldots,F_n(z))$.
We consider the mapping $F_j: B_X\to\ol{\mathbb{D}}$ defined by
\beas F_j(z)=a_j+\sum_{s=1}^\infty (P_{s})_j (z),\quad \text{where}\quad z\in B_X\quad\text{and}\quad j\in\{1,2\ldots,n\}.\eeas
In view of the relation (\ref{e14}), we consider the following integral for the holomorphic mapping $F_j$ with values in $\ol{\mathbb{D}}$:
\beas\int_{0}^1\frac{F_j(t z)}{1-t z}\; d t&=&\int_{0}^1\frac{a_j}{1-t z}\; dt+\sum_{s=1}^\infty (P_{s})_j (z) \int_{0}^1\frac{t^s}{1-t z}\; dt\nonumber\\[2mm]
&=&a_j\left(1+\frac{z}{2}+\frac{z^2}{3}+\cdots\right)+\sum_{s=1}^\infty (P_{s})_j (z)\int_{0}^1 t^s\left(1+t z+t^2 z^2+\cdots\right)\; dt\nonumber\\[2mm]
&=&\sum_{n=0}^\infty\frac{a_j z^n}{n+1}+(P_{1})_j (z)\left(\frac{1}{2}+\frac{z}{3}+\frac{z^2}{4}+\cdots\right)+(P_{2})_j (z)\left(\frac{1}{3}+\frac{z}{4}\right.\nonumber\\
&&\left.+\frac{z^2}{5}+\cdots\right)+(P_{3})_j (z)\left(\frac{1}{4}+\frac{z}{5}+\frac{z^2}{6}+\cdots\right)+\cdots\nonumber\\[2mm]
&=&\sum_{n=0}^\infty\frac{a_j z^n}{n+1}+\frac{1}{2}(P_{1})_j (z)+\frac{1}{3}\left((P_{2})_j (z)+z(P_{1})_j (z)\right)\nonumber\\[2mm]
&&+\frac{1}{4}\left((P_{3})_j (z)+z^2(P_{2})_j (z)+z^3 (P_{1})_j (z)\right)+\cdots.\eeas
Therefore, the majorant series corresponding to the Ces\'aro operator for a holomorphic mapping $F_j$ is as follows:
\beas \sum_{n=0}^\infty\frac{|a_j| \Vert z\Vert^n}{n+1}+\sum_{n=1}^\infty \frac{1}{n+1}\left(\sum_{k=1}^n\Vert z\Vert^{k-1}  \left|\left(P_{n-k+1}\right)_j(z)\right|\right).\eeas
As $ \Vert P_s(z)\Vert_\infty=\max_{1\leq j\leq n} |(P_{s})_j(z)|$, thus, the majorant series corresponding to the Ces\'aro operator for a holomorphic mapping $F: B_X\to\ol{\mathbb{D}^n}$ is 
\beas \mathcal{C}_F(\Vert z\Vert):=\sum_{n=0}^\infty\frac{\Vert a\Vert_\infty \Vert z\Vert^n}{n+1}+\sum_{n=1}^\infty \frac{1}{n+1}\left(\sum_{k=1}^n\Vert z\Vert^{k-1}  \left\Vert\left(P_{n-k+1}\right)(z)\right\Vert_\infty\right).\eeas
In the following result, we establish the Bohr-type inequality for the Ces\'aro operator defined on a set of holomorphic mappings $F:B_X\to\ol{\mathbb{D}^n}$.
\begin{theo}\label{Th1}
Let $B_X$ be the unit ball of a complex Banach space $X$. Let $a=(a_1, a_2, \ldots, a_n)\in\ol{\mathbb{D}^n}$ and let $F: B_X \to\ol{\mathbb{D}^n}$ be a holomorphic mapping with 
\beas F(z)=a+\sum_{s=1}^\infty P_s(z),\quad z\in B_X,\eeas
where $P_s(z)=(1/s!) D^s F(0)(z^s)$. Assume that $|a_j|=\Vert a\Vert_\infty$ for all $j\in\{1, 2, \ldots,n\}$. Then, we have 
\bea\label{e1}\mathcal{C}_F(\Vert z\Vert)\leq \frac{1}{r}\log\frac{1}{1-r}\quad\text{for}\quad \Vert z\Vert=r\leq R,  \eea
where $R\in(0, 1)$ is the positive root of the equation 
\beas 3(1-r )\log\frac{1}{1-r}-2r=0.\eeas
The number $R$ cannot be improved.
 \end{theo}
\begin{proof}
Let $z_0\in\pa B_X$ be fixed. Since $F=(F_1,F_2,\cdots,F_n)$, we consider $f_j(t)=F_j(t z_0)$ for $t\in \mathbb{D}$. Then, $f_j\in H(\mathbb{D}, \ol {\mathbb{D}})$ and 
\beas f_j(t)=a_j+\sum_{s=1}^\infty (P_{s})_j (z_0)t^s,\quad t\in\mathbb{D}.\eeas
In view of \textrm{Lemma \ref{lem}}, we have
\beas |(P_{s})_j(z_0)|\leq 1-|a_j|^2=1-\Vert a\Vert^2_{\infty}\quad \text{for}\quad s\geq 1\eeas
for each fixed $j$ $(j=1,2,\cdots)$.
Since $ \Vert P_s(z_0)\Vert_{\infty}=\max\limits_{j} |(P_{s})_j(z_0)|$, thus, we have
\bea\label{a20}  \Vert P_s(z_0)\Vert_{\infty}\leq 1-\Vert a\Vert^2_{\infty}\quad \text{for}\quad s\geq 1.\eea
Let $x=\Vert a\Vert_{\infty}\in[0,1]$. It is evident that, if $x=1$, then $P_s(z_0)=0$ for $s\geq 1$. Therefore, the inequality (\ref{e1}) holds for $z=rz_0$ and $r\in[0,1)$. Thus, 
we may assume that $x\in[0,1)$. Using (\ref{a20}), for $r\in(0,1)$, we have
\beas 
&&\sum_{n=0}^\infty \frac{x }{n+1}\Vert rz_0\Vert^n+\sum_{n=1}^\infty \frac{1}{n+1}\left(\sum_{k=1}^n  \Vert rz_0\Vert^{k-1} \Vert P_{n-k+1}(rz_0)\Vert_\infty\right)\\[2mm]
&\leq&x\sum_{n=0}^\infty \frac{r^n}{n+1}+\sum_{n=1}^\infty\frac{1}{n+1}\left((1-x^2) \sum_{k=1}^n r^n\right)\\
&=&\frac{x}{r}\log\frac{1}{1-r}+(1-x^2)\sum_{n=1}^\infty \frac{nr^n}{n+1}\\
&=& \frac{x}{r}\log\frac{1}{1-r}+(1-x^2)\left(\frac{1}{1-r}- \frac{1}{r}\log\frac{1}{1-r}\right):=G(x,r).\eeas
%Let 
%\beas G(x,r)=\frac{x}{r}\log\frac{1}{1-r}+(1-x^2)\left(\frac{1}{1-r}- \frac{1}{r}\log\frac{1}{1-r}\right),~ \text{where}~ x\in[0,1),~r\in(0,1). \eeas
Differentiating partially $G(x,r)$ twice with respect to $x$, we obtain
\beas &&\frac{\pa}{\pa x}G(x,r)=\frac{1}{r}\log\frac{1}{1-r}-2x\left(\frac{1}{1-r}- \frac{1}{r}\log\frac{1}{1-r}\right)\\[2mm]\text{and}
&&\frac{\pa^2}{\pa x^2}G(x,r)=-2\left(\frac{1}{1-r}- \frac{1}{r}\log\frac{1}{1-r}\right)\leq 0,\eeas
which shows that $\frac{\pa}{\pa x}G(x,r)$ is a monotonically decreasing function of $x\in[0, 1)$ and it follows that 
\beas \frac{\pa}{\pa x}G(x,r)\geq \lim_{x\to1^-}\frac{\pa}{\pa x}G(x,r)=\frac{1}{r(1-r)}\left(-2r +3(1-r )\log\frac{1}{1-r}\right)\geq 0\eeas
for $r\leq R$, where $R(\approx 0.533589)\in(0, 1)$ is the unique positive root of the equation
\beas G_1(r):=-2r +3(1-r )\log\frac{1}{1-r}.\eeas
\begin{figure}[H]
\centering
\includegraphics[scale=1]{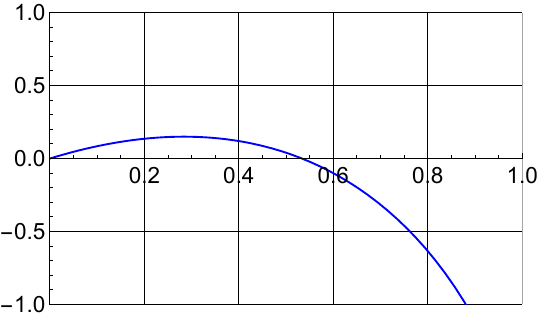}
\caption{The graph of $G_1(r)$ in (0, 1)}
\label{fig1}
\end{figure}
\noindent Thus, $G(x,r)$ is a monotonically increasing function of $x\in[0, 1)$ for $r\leq R$, and it follows that
\beas G(x,r)\leq \lim_{x\to 1^-}G(x,r)=\frac{1}{r}\log\frac{1}{1-r}\quad\text{for all}\quad r\leq R.\eeas
\indent To prove the sharpness of the constant $R$, we consider the function
\beas F_1(z)=f(l_{z_0}(z))u,\quad z\in B_X,\eeas
where $z_0\in \pa B_X$, $l_{z_0}\in T(z_0)$, $u=(u_1,u_2,\dots,u_n)\in \pa \mathbb{D}^n$ with $|u_1|=|u_2|=\dots =|u_n|=1$, and 
\beas f(t)= \frac{\lambda -t}{1-\lambda t}, \quad t\in \mathbb{D},~\lambda \in (0,1).\eeas
Then, we have
\beas F_1(rz_0)=f\left(l_{z_0}(rz_0)\right)u = f\left(r \; l_{z_0}(z_0)\right) u= f(r)u= a+\sum_{s=1}^\infty P_s(rz_0),\eeas
where $a=\lambda u$ and $P_s(z_0)=\lambda^{s-1}(\lambda^2 -1)u$ for $s\geq 1$. Therefore,
\bea &&\sum_{n=0}^\infty \frac{\Vert \lambda u\Vert_\infty }{n+1}\Vert rz_0\Vert^n+\sum_{n=1}^\infty \frac{1}{n+1}\left(\sum_{k=1}^n  \Vert rz_0\Vert^{k-1} \Vert P_{n-k+1}(rz_0)\Vert_\infty\right)\nonumber\\[2mm]
&=&\sum_{n=0}^\infty \frac{\lambda r^n}{n+1} +\sum_{n=1}^\infty \frac{1}{n+1}\left(\sum_{k=1}^n r^n \lambda^{n-k}(1-\lambda^2)\right)\nonumber\\[2mm]
&=&\frac{\lambda}{r}\log{\frac{1}{1-r}}+(1-\lambda^2)\sum_{n=1}^\infty \frac{r^n}{n+1}\left(\sum_{k=1}^n \lambda^{n-k}\right)\nonumber\\[2mm]
&=& \frac{\lambda}{r}\log{\frac{1}{1-r}}+(1-\lambda^2)\sum_{n=1}^\infty\left( \frac{r^n}{n+1}\cdot\frac{1-\lambda^n}{1-\lambda}\right)\nonumber\eea
\bea&=& \frac{\lambda}{r}\log{\frac{1}{1-r}}+(1+\lambda)\sum_{n=1}^\infty \left(\frac{r^n}{n+1}-\frac{(\lambda r)^n}{n+1}\right)\nonumber\hspace{3cm}\\[2mm]
&=&\frac{1}{r}\log{\frac{1}{1-r}}+\frac{2\lambda}{r}\log{\frac{1}{1-r}}-\frac{1+\lambda}{\lambda r}\log{\frac{1}{1-\lambda r}}\nonumber\\[2mm]
\label{e3}&=& \frac{1}{r}\log{\frac{1}{1-r}} +(1-\lambda)\frac{3(1-r)\log(1-r)+2r}{r(1-r)}+D_{\lambda}(r),\eea
where
\bea D_{\lambda}(r)&=& \frac{3-\lambda}{r}\log{\frac{1}{1-r}}-\frac{2(1-\lambda)}{1-r}-\frac{1+\lambda}{\lambda r}\log{\frac{1}{1-\lambda r}}\nonumber\\
&=&\sum_{n=2}^\infty \left(\frac{3-\lambda}{n}-2(1-\lambda)-\frac{(1+\lambda)\lambda^{n-1}}{n}\right)r^{n-1}\nonumber\\
\label{e4}&= & O\left((1-\lambda)^2\right) \quad \text{as} \quad \lambda\to 1^-.\eea
Furthermore, it is evident that for $r>R$, the following inequality holds
\bea\label{e5} \frac{3(1-r)\log(1-r)+2r}{r(1-r)}>0.\eea
From (\ref{e3}), (\ref{e4}) and (\ref{e5}), we conclude that $R$ cannot be improved. This completes the proof.
\end{proof}
Before proceeding to our next result, it is necessary to establish the majorant series corresponding to the Bernardi operator for a holomorphic mapping $F:B_X\to\ol{\mathbb{D}^n}$ with $F(z)=\sum_{s=m}^\infty P_s(z)$, where $m\in\mathbb{N}\cup\{0\}$, 
$z\in B_X$, $P_s(z)=(1/s!) D^s F(0)(z^s)$ and $F(z)=(F_1(z),F_2(z),\ldots,F_n(z))$.
We consider the mapping $F_j: B_X\to\ol{\mathbb{D}}$ defined by
\beas F_j(z)=\sum_{s=m}^\infty (P_{s})_j (z),\quad \text{where}\quad z\in B_X\quad\text{and}\quad j\in\{1,2\ldots,n\}.\eeas
In view of the relation (\ref{e16}), we consider the following integral for the holomorphic mapping $F_j: B_X\to \ol{\mathbb{D}}$: 
\beas (1+\beta)\int_0^1 F_j(t z) t^{\beta-1}\; d t=(1+\beta)\sum_{s=m}^\infty (P_{s})_j (z) \int_{0}^1 t^{s+\beta-1}\; dt
=(1+\beta)\sum_{s=m}^\infty \frac{(P_{s})_j (z)}{s+\beta}.\eeas
It is evident that, the majorant series corresponding to the Bernardi operator for a holomorphic mapping $F:B_X\to\ol{\mathbb{D}^n}$ is 
\beas (1+\beta)\sum_{s=m}^\infty \frac{\Vert P_{s} (z)\Vert_\infty}{s+\beta}.\eeas
In the following result, we establish the Bohr-type inequality for the Bernardi operator defined on a set of holomorphic mappings $F:B_X\to\ol{\mathbb{D}^n}$.
\begin{theo}\label{Th2}
Let $B_X$ be the unit ball of a complex Banach space $X$, $m\in\mathbb{N}\cup\{0\}$ and $\beta>-m$. Let $F: B_X \to\ol{\mathbb{D}^n}$ be a holomorphic mapping with 
\beas F(z)=P_m(z)+\sum_{s=m+1}^\infty P_s(z),\quad z\in B_X,\eeas
where $P_s(z)=(1/s!) D^s F(0)(z^s)$. Assume that $\Vert P_m(z)\Vert_\infty$ is constant on $\pa B_X$ and $|(P_m)_j(z_0)|=\Vert P_m(z_0)\Vert_\infty$ for all $j\in\{1, 2, \ldots,n\}$ and $z_0\in\pa B_X$. Then, we have 
\bea\label{e6} \sum_{s=m}^\infty \frac{\left\Vert P_{s}(z)\right\Vert_\infty }{\beta+s} \leq \frac{r^m}{m+ \beta}\eea
for $\Vert z\Vert=r\leq R(\beta)$, where $R(\beta)\in(0, 1)$ is the unique root of the equation 
\beas \frac{1}{\beta+m}-2\sum_{s=m+1}^\infty \frac{1}{\beta+s}r^{s-m}=0.\eeas
The number $R(\beta)$ cannot be improved.
 \end{theo}
\begin{proof}
Let $z_0\in\pa B_X$ be fixed. Since $F=(F_1,F_2,\cdots,F_n)$, let $f_j(t)=F_j(t z_0)$ for $t\in \mathbb{D}$. Then, $f_j\in H(\mathbb{D}, \ol {\mathbb{D}})$ and 
\beas f_j(t)=t^m\left((P_m)_j(z_0)+\sum_{s=m+1}^\infty (P_{s})_j(z_0)t^{s-m}\right),\quad t\in\mathbb{D}.\eeas
Let $g_j(t)=(P_m)_j(z_0)+\sum_{s=m+1}^\infty (P_{s})_j(z_0)t^{s-m}$ for $t\in\mathbb{D}$. Then, $g_j(t)\in H(\mathbb{D}, \ol{\mathbb{D}})$.
In view of \textrm{Lemma \ref{lem}}, for each fixed $j$ $(j=1,2,\cdots, n)$, we have
\beas |(P_{s})_j(z_0)|\leq 1-|(P_m)_j(z_0)|^2=1-\Vert P_m(z_0)\Vert^2_{\infty}\quad\text{for}\quad s\geq m+1.\eeas
Thus, we have
\bea\label{e7}  \Vert P_s(z_0)\Vert_{\infty}\leq 1-\Vert P_m(z_0)\Vert^2_{\infty}\quad\text{for}\quad s\geq m+1.\eea
 Let $x=\Vert P_m(z_0)\Vert_{\infty}\in[0,1]$. If $x=1$, then $P_s(z_0)=0$ for all $s\geq m+1$.
 Therefore, the inequality (\ref{e6}) holds for $z=rz_0$ and $r\in[0,1)$. We assume that $x\in[0,1)$. For $r\in(0,1)$ and by using (\ref{e7}), we have
\beas\sum_{s=m}^\infty \frac{\left\Vert P_{s}(rz_0)\right\Vert_\infty}{\beta+s}&\leq& \frac{\Vert P_m(z_0)\Vert_{\infty}}{\beta+m}r^m+(1-\Vert P_m(z_0)\Vert_{\infty}^2)\sum_{s=m+1}^\infty \frac{1}{\beta+s}r^s\\
&=& \frac{x}{\beta+m}r^m+(1-x^2)\sum_{s=m+1}^\infty \frac{1}{\beta+s}r^s:=G_2(x, r).\eeas
Differentiating partially $G_2(x,r)$ twice with respect to $x$, we obtain
\beas &&\frac{\pa}{\pa x}G_2(x,r)=\frac{1}{\beta+m}r^m-2x\sum_{s=m+1}^\infty \frac{1}{\beta+s}r^s\\[2mm]\text{and}
&&\frac{\pa^2}{\pa x^2}G_2(x,r)=-2\sum_{s=m+1}^\infty \frac{1}{\beta+s}r^s\leq 0.\eeas
Thus, $\frac{\pa}{\pa x}G_2(x,r)$ is a monotonically decreasing function of $x\in[0, 1)$ and it follows that
\beas \frac{\pa}{\pa x}G_2(x,r)\geq \lim_{x\to1^-}\frac{\pa}{\pa x}G_2(x,r)=\frac{1}{\beta+m}r^m-2\sum_{s=m+1}^\infty \frac{1}{\beta+s}r^s\geq 0\eeas
for $r\leq R(\beta)$, where $R(\beta)\in(0, 1)$ is the unique positive root of the equation
\beas G_3(r):=\frac{1}{\beta+m}-2\sum_{s=m+1}^\infty \frac{1}{\beta+s}r^{s-m}= 0.\eeas
It is evident that $G_3(r)$ is a monotonically decreasing function of $r\in(0,1)$ and $\lim_{r\to0^+} G_3(r)=1/(\beta+m)>0$. 
Therefore, $G_2(x,r)$ is a monotonically increasing function of $x$ for $r\leq R(\beta)$ and it follows that
\beas \sum_{s=m}^\infty \frac{\left\Vert P_{s}(rz_0)\right\Vert_\infty}{\beta+s}\leq G_2(x,r) \leq \lim_{x\to1^-}G_2(x,r)=\frac{1}{m+\beta}r^m \quad \text{for}\quad r\leq R(\beta).\eeas
\indent To prove the sharpness of the constant $R(\beta)$, we consider the function
\beas F_2(z)=f(l_{z_0}(z))u,\quad z\in B_X,\eeas
where $z_0\in \pa B_X$,  $l_{z_0}\in T(z_0)$, $u=(u_1,u_2,\dots,u_n)\in \pa \mathbb{D}^n$ with $|u_1|=|u_2|=\dots =|u_n|=1$, and
\beas f(t)= t^m\frac{\lambda -t}{1-\lambda t}, \quad t\in \mathbb{D},~\lambda \in (0,1).\eeas
Thus, we have
\beas F_2(rz_0)=f(l_{z_0}(rz_0))u = \sum_{s=m}^\infty P_s(rz_0),\eeas
where $P_m(rz_0)=\lambda r^mu$ and $P_s(rz_0)=\lambda^{s-1}(\lambda^2 -1)r^{s}u$ for $s\geq m+1$. 
Therefore, 
\bea\label{e8}\sum_{s=m}^\infty \frac{\left\Vert P_{s}(rz_0)\right\Vert_\infty}{\beta+s}&=&\frac{\lambda}{\beta+m}r^m+(1-\lambda^2)\sum_{s=m+1}^\infty \frac{\lambda^{s-1}}{\beta+s}r^{s}\\[2mm]
&=&\frac{1}{\beta+m}r^m-(1-\lambda)r^m\left(\frac{1}{\beta+m}-2\sum_{s=m+1}^\infty \frac{1}{\beta+s}r^{s-m}\right)+H_{\lambda}(r),\nonumber\eea
where
\beas H_{\lambda}(r)=2(\lambda-1)\sum_{s=m+1}^\infty \frac{1}{\beta+s}r^s +(1-\lambda^2)\sum_{s=m+1}^\infty \frac{\lambda^{s-1}}{\beta+s}r^s.\eeas
Letting $\lambda\to 1^-$, we obtain
\bea\label{e9} H_\lambda(r)=\sum_{s=m+1}^\infty \frac{2(\lambda-1)+(1-\lambda^2)\lambda^{s-1}}{\beta+s}r^s=O((1-\lambda)^2).\eea
Furthermore, for $r>R(\beta)$, we have the following inequality
\bea\label{e10} \frac{1}{\beta+m}-2\sum_{s=m+1}^\infty \frac{1}{\beta+s}r^{s-m}<0.\eea
From (\ref{e8}), (\ref{e9}) and (\ref{e10}), we conclude that $R(\beta)$ cannot be improved. This completes the proof.
\end{proof}
In view of (\ref{e15}), we consider the majorant series corresponding to the discrete Fourier transform for a holomorphic mapping $F:B_X\to\ol{\mathbb{D}^n}$ with $F(z)=a+\sum_{s=1}^\infty P_s(z)$, where $z\in B_X$ and
 $P_s(z)=(1/s!) D^s F(0)(z^s)$ as follows:
 \beas\mathcal{F}_F(\Vert z\Vert):=\sum_{n=0}^\infty \Vert a\Vert_\infty\Vert z\Vert^n+\sum_{n=1}^\infty\left(\sum_{k=1}^n\Vert z\Vert^{k-1}  \Vert P_{n-k+1}(z)\Vert_\infty\right).\eeas
In the following result, we establish the Bohr-type inequality for the discrete Fourier transform defined on a set of holomorphic mappings $F:B_X\to\ol{\mathbb{D}^n}$.
\begin{theo}\label{Th3}
Let $B_X$ be the unit ball of a complex Banach space $X$. Let $a=(a_1, a_2, \ldots, a_n)\in\ol{\mathbb{D}^n}$ and let $F: B_X \to\ol{\mathbb{D}^n}$ be a holomorphic mapping with 
\beas F(z)=a+\sum_{s=1}^\infty P_s(z),\quad z\in B_X,\eeas
where $P_s(z)=(1/s!) D^s F(0)(z^s)$. Assume that $|a_j|=\Vert a\Vert_\infty$ for all $j\in\{1, 2, \ldots,n\}$. Then, we have 
\bea\label{f1}\mathcal{F}_F(\Vert z\Vert)\leq \frac{1}{1-r}\quad\text{for}\quad \Vert z\Vert=r\leq 1/3. \eea
 The number $1/3$ cannot be improved.
 \end{theo}
\begin{proof}
Let $z_0\in\pa B_X$ be fixed. Since $F=(F_1,F_2,\cdots,F_n)$, we consider $f_j(t)=F_j(t z_0)$ for $t\in \mathbb{D}$. Then, $f_j\in H(\mathbb{D}, \ol {\mathbb{D}})$ and 
\beas f_j(t)=a_j+\sum_{s=1}^\infty (P_{s})_j(z_0)t^s,\quad t\in\mathbb{D}.\eeas
Using similar argument as in the proof of \textrm{Theorem \ref{Th1}}, we have
\bea\label{f2}  \Vert P_s(z_0)\Vert_{\infty}\leq 1-\Vert a\Vert^2_{\infty}\quad \text{for}\quad s\geq 1.\eea
Let $x=\Vert a\Vert_{\infty}\in[0,1]$. It is evident that, if $x=1$, then $P_s(z_0)=0$ for all $s\geq 1$. Therefore, the inequality (\ref{f1}) holds for $z=rz_0$ and $r\in[0,1)$. Thus, 
we may assume that $x\in[0,1)$. Using (\ref{f2}), for $r\in(0,1)$, we have
\beas 
&&\sum_{n=0}^\infty x\Vert rz_0\Vert^n+\sum_{n=1}^\infty\left(\sum_{k=1}^n  \Vert rz_0\Vert^{k-1} \Vert P_{n-k+1}(rz_0)\Vert_\infty\right)\\[2mm]
&\leq&x\sum_{n=0}^\infty r^n+\sum_{n=1}^\infty\left((1-x^2) \sum_{k=1}^n r^n\right)\\
&=&\frac{x}{1-r}+(1-x^2)\sum_{n=1}^\infty nr^n\\
&=& \frac{x}{1-r}+(1-x^2)\frac{r}{(1-r)^2}:=G_4(x,r).\eeas
By partially differentiating $G_4(x,r)$ twice with respect to $x$, we obtain
\beas\frac{\pa}{\pa x}G_4(x,r)=\frac{1}{1-r}-2x\frac{r}{(1-r)^2}\quad\text{and}\quad \frac{\pa^2}{\pa x^2}G_4(x,r)=-2\frac{r}{(1-r)^2}\leq 0,\eeas
which shows that $\frac{\pa}{\pa x}G(x,r)$ is a monotonically decreasing function of $x\in[0, 1)$ and it follows that 
\beas \frac{\pa}{\pa x}G_4(x,r)\geq \lim_{x\to1^-}\frac{\pa}{\pa x}G_4(x,r)=\frac{1-3r}{(1-r)^2}\geq 0\eeas
for $r\leq 1/3$.
Thus, $G_4(x,r)$ is a monotonically increasing function of $x\in[0, 1)$ for $r\leq 1/3$ and it follows that 
\beas G_4(x,r)\leq \lim_{x\to 1^-}G_4(x,r)=\frac{1}{1-r}\quad\text{for}\quad r\leq \frac{1}{3}.\eeas
\indent To prove the sharpness of the constant $1/3$, we consider the function
\beas F_3(z)=f(l_{z_0}(z))u,\quad z\in B_X,\eeas
where $z_0\in \pa B_X$, $l_{z_0}\in T(z_0)$, $u=(u_1,u_2,\dots,u_n)\in \pa \mathbb{D}^n$ with $|u_1|=|u_2|=\dots =|u_n|=1$, and 
\beas f(t)= \frac{\lambda -t}{1-\lambda t}, \quad t \in \mathbb{D},~\lambda \in (0,1).\eeas
Thus, we have
\beas F_3(rz_0)=f(l_{z_0}(rz_0))u = a+\sum_{s=1}^\infty P_s(rz_0),\eeas
where $a=\lambda u$ and $P_s(z_0)=\lambda^{s-1}(\lambda^2 -1)u$ for $s\geq 1$. Therefore,
\bea &&\sum_{n=0}^\infty \Vert \lambda u\Vert_\infty \Vert rz_0\Vert^n+\sum_{n=1}^\infty \left(\sum_{k=1}^n  \Vert rz_0\Vert^{k-1} \Vert P_{n-k+1}(rz_0)\Vert_\infty\right)\nonumber\\[2mm]
&=&\sum_{n=0}^\infty \lambda r^n +\sum_{n=1}^\infty\left(\sum_{k=1}^n r^n \lambda^{n-k}(1-\lambda^2)\right)\nonumber\\[1mm]
&=&\frac{\lambda}{1-r}+(1-\lambda^2)\sum_{n=1}^\infty r^n\left(\sum_{k=1}^n \lambda^{n-k}\right)\nonumber\\[1mm]
&=& \frac{\lambda}{1-r}+(1-\lambda^2)\sum_{n=1}^\infty r^n\cdot\frac{1-\lambda^n}{1-\lambda}\nonumber\\[1mm]
%&=& \frac{\lambda}{1-r}+(1+\lambda)\sum_{n=1}^\infty \left(r^n-(\lambda r)^n\right)\nonumber\\[2mm]
&=& \frac{\lambda}{1-r}+(1+\lambda)\left(\frac{r}{1-r}-\frac{\lambda r}{1-\lambda r}\right)\nonumber\\[1mm]
\label{f3}&=&\frac{1}{1-r}-(1-\lambda)\frac{1-3r}{1-r}+G_5(\lambda, r),\eea
where
\beas G_5(\lambda, r)&=& \frac{(1+\lambda)}{(1-r)}\left(r-\frac{\lambda r(1-r)}{1-\lambda r}\right)+(1-\lambda)\frac{1-3r}{1-r}-\frac{1-\lambda}{1-r}\\
&=& \frac{(1+\lambda)}{(1-r)}\left(\frac{r(1-\lambda)}{1-\lambda r}\right)-(1-\lambda)\frac{3r}{1-r}\\
&=& (1-\lambda)\frac{r}{1-r}\left(\frac{\lambda+3\lambda r-2}{1-\lambda r}\right)~\text{which tends to}~0~\text{as}~\lambda\to1^-.\eeas
Furthermore, for $r>1/3$, we have $(1-\lambda)(1-3r)/(1-r)<0$.
From (\ref{f3}), we conclude that $1/3$ cannot be improved. This completes the proof.
\end{proof}
Using similar argument to those used for establishing the majorant series corresponding to the Ces\'aro operator for a holomorphic mapping $F:B_X\to\ol{D}^n$ and in view of (\ref{g6}), we obtain
\bs\beas \int_{0}^1\frac{F_j(t z)}{(1-t z)^{\beta}} d t= \sum_{n=0}^\infty \frac{a_j\Gamma(n+\beta) }{\Gamma(n+2)\Gamma(\beta)} z^n+\sum_{n=1}^\infty \frac{1}{n+1}\left(\sum_{k=1}^n\frac{\Gamma(k-1+\beta)}{\Gamma(k)\Gamma(\beta)} z^{k-1} (P_{n-k+1})_j(z)\right),\eeas\es
where $F_j: B_X\to\ol{\mathbb{D}}$ is a holomorphic mapping such that $F_j(z)=a_j+\sum_{s=1}^\infty (P_{s})_j (z)$, $z\in B_X$ and  $j\in\{1,2\ldots,n\}$.
Evidently, the majorant series corresponding to the $\beta$-Ces\'aro operator for a holomorphic mapping $F:B_X\to\ol{\mathbb{D}^n}$ is given by
\bs\beas \mathcal{C}_\beta^F(\Vert z\Vert):=\sum_{n=0}^\infty \frac{\Vert a\Vert_\infty\Gamma(n+\beta) }{\Gamma(n+2)\Gamma(\beta)}\Vert z\Vert^n+\sum_{n=1}^\infty \frac{1}{n+1}\left(\sum_{k=1}^n\frac{\Gamma(k-1+\beta)}{\Gamma(k)\Gamma(\beta)}\Vert z\Vert^{k-1}  \Vert P_{n-k+1}(z)\Vert_\infty\right).\eeas\es
In the following result, we establish the Bohr-type inequality for the $\beta$-Ces\'aro operator defined on a set of holomorphic mappings $F:B_X\to\ol{\mathbb{D}^n}$.
\begin{theo}\label{Th4}
Let $B_X$ be the unit ball of a complex Banach space $X$ and $0<\beta(\not=1)$. Let $a=(a_1, a_2, \ldots, a_n)\in\ol{\mathbb{D}^n}$ and let $F: B_X \to\ol{\mathbb{D}^n}$ be a holomorphic mapping with 
\beas F(z)=a+\sum_{s=1}^\infty P_s(z),\quad z\in B_X,\eeas
where $P_s(z)=(1/s!) D^s F(0)(z^s)$. Assume that $|a_j|=\Vert a\Vert_\infty$ for all $j\in\{1, 2, \ldots,n\}$. Then, we have 
\bea\label{e11}\mathcal{C}_\beta^F(\Vert z\Vert)\leq \frac{1}{r}\left(\frac{1-(1-r)^{1-\beta}}{1-\beta}\right)\quad \text{for}\quad \Vert z\Vert=r\leq R_\beta,\eea
 where $R_\beta\in(0, 1)$ is the positive root of the equation 
\beas \frac{3\left(1-(1-r)^{1-\beta}\right)}{1-\beta}-\frac{2\left((1-r)^{-\beta}-1\right)}{\beta}=0.\eeas
The number $R_\beta$ cannot be improved.
 \end{theo}
\begin{proof}
Let $z_0\in\pa B_X$ be fixed. Since $F=(F_1,F_2,\cdots,F_n)$, we consider the function $f_j(t)=F_j(t z_0)$ for $t\in \mathbb{D}$. Then, $f_j\in H(\mathbb{D}, \ol {\mathbb{D}})$ and 
\beas f_j(t)=a_j+\sum_{s=1}^\infty (P_{s})_j(z_0)t^s,\quad t\in\mathbb{D}.\eeas
Using similar argument as in the proof of \textrm{Theorem \ref{Th1}}, we have
\bea\label{g1}  \Vert P_s(z_0)\Vert_{\infty}\leq 1-\Vert a\Vert^2_{\infty}\quad \text{for}\quad s\geq 1.\eea
Let $x=\Vert a\Vert_{\infty}\in[0,1]$. Evidently, if $x=1$, then $P_s(z_0)=0$ for all $s\geq 1$. Therefore, the inequality (\ref{e11}) holds for $z=rz_0$ and $r\in[0,1)$. Let $x\in[0,1)$. Using (\ref{ceq}), (\ref{ceq1}) and (\ref{g1}), for $r\in(0,1)$, we have
\beas &&x\sum_{n=0}^\infty \frac{1}{n+1} \frac{\Gamma(n+\beta) }{\Gamma(n+1)\Gamma(\beta)}\Vert rz_0\Vert^n\\
&&+\sum_{n=1}^\infty \frac{1}{n+1}\left(\sum_{k=1}^n\frac{\Gamma(k-1+\beta)}{\Gamma(k)\Gamma(\beta)}\Vert rz_0\Vert^{k-1}  \Vert P_{n-k+1}(rz_0)\Vert_\infty\right)\\[1mm]
&\leq&
%x\sum_{n=0}^\infty \frac{1}{n+1}\frac{\Gamma(n+\beta) }{\Gamma(n+1)\Gamma(\beta)}r^n+\sum_{n=1}^\infty \frac{1}{n+1}\left((1-x^2)\sum_{k=1}^n\frac{\Gamma(k-1+\beta)}{\Gamma(k)\Gamma(\beta)}r^n\right)\\[2mm]
x\sum_{n=0}^\infty\frac{1}{n+1} \frac{\Gamma(n+\beta) }{\Gamma(n+1)\Gamma(\beta)}r^n+(1-x^2)\sum_{n=1}^\infty \frac{r^n}{n+1}\left(\sum_{k=1}^n\frac{\Gamma(k-1+\beta)}{\Gamma(k)\Gamma(\beta)}\right)\\[1mm]
&=&x\sum_{n=0}^\infty \frac{1}{n+1}\frac{\Gamma(n+\beta) }{\Gamma(n+1)\Gamma(\beta)}r^n+(1-x^2)\sum_{n=1}^\infty \frac{r^n}{n+1}\left(\sum_{k=1}^n\frac{\Gamma(n-k+\beta)}{\Gamma(n-k+1)\Gamma(\beta)}\right)\\
&=&\frac{x}{r}\int_{0}^r\frac{1}{(1-t)^{\beta}} d t+\frac{1-x^2}{r}\int_{0}^r\frac{t}{(1-t)^{\beta+1}} d t\\
&=&\frac{x^2+x-1}{r}\int_{0}^r\frac{1}{(1-t)^{\beta}} d t+\frac{1-x^2}{r}\int_{0}^r\frac{1}{(1-t)^{\beta+1}} d t\eeas
\beas&=&\frac{1}{r}\left(\frac{(x^2+x-1)\left(1-(1-r)^{1-\beta}\right)}{1-\beta}+\frac{(1-x^2)\left((1-r)^{-\beta}-1\right)}{\beta}\right
):=G_6(x,r).\eeas
By partially differentiating $G_6(x,r)$ twice with respect to $x$, we obtain
\beas &&\frac{\pa}{\pa x}G_6(x,r)=\frac{1}{r}\left(\frac{(2x+1)\left(1-(1-r)^{1-\beta}\right)}{1-\beta}-\frac{2x\left((1-r)^{-\beta}-1\right)}{\beta}\right
)\\[1mm]\text{and}
&&\frac{\pa^2}{\pa x^2}G_6(x,r)=\frac{1}{r}\left(\frac{2\left(1-(1-r)^{1-\beta}\right)}{1-\beta}-\frac{2\left((1-r)^{-\beta}-1\right)}{\beta}\right
)\leq 0,\eeas
which shows that $\frac{\pa}{\pa x}G_6(x,r)$ is a monotonically decreasing function of $x\in[0, 1)$ and it follows that 
\beas \frac{\pa}{\pa x}G_6(x,r)\geq \lim_{x\to1^-}\frac{\pa}{\pa x}G(x,r)=\frac{1}{r}\left(\frac{3\left(1-(1-r)^{1-\beta}\right)}{1-\beta}-\frac{2\left((1-r)^{-\beta}-1\right)}{\beta}\right)\geq 0\eeas
for $r\leq R_\beta$, where $R_\beta\in (0, 1)$ is the positive root of the equation 
\beas \frac{3\left(1-(1-r)^{1-\beta}\right)}{1-\beta}-\frac{2\left((1-r)^{-\beta}-1\right)}{\beta}=0.\eeas
Thus, $G_6(x,r)$ is a monotonically increasing function of $x\in[0, 1)$ for $r\leq R_\beta$ and it follows that
\beas G_6(x,r)\leq \lim_{x\to 1^-}G_6(x,r)=\frac{1}{r}\left(\frac{1-(1-r)^{1-\beta}}{1-\beta}\right)\quad\text{for}\quad r\leq R_\beta.\eeas
\indent To prove the sharpness of the constant $R_\beta$, we consider the function
\beas F_4(z)=f(l_{z_0}(z))u,\quad z\in B_X,\eeas
where $z_0\in \pa B_X$, $l_{z_0}\in T(z_0)$, $u=(u_1,u_2,\dots,u_n)\in \pa \mathbb{D}^n$ with $|u_1|=|u_2|=\dots =|u_n|=1$, and 
\beas f(t)= \frac{\lambda -t}{1-\lambda t}, \quad t \in \mathbb{D},~\lambda \in (0,1).\eeas
Then, we have $F_4(rz_0)=f(l_{z_0}(rz_0))u = a+\sum_{s=1}^\infty P_s(rz_0)$,
where $a=\lambda u$ and $P_s(z_0)=\lambda^{s-1}(\lambda^2 -1)u$ for $s\geq 1$. 
Using (\ref{ceq}) and (\ref{ceq1}), we have
\bea &&\Vert \lambda u\Vert_\infty\sum_{n=0}^\infty \frac{1}{n+1}\frac{\Gamma(n+\beta) }{\Gamma(n+1)\Gamma(\beta)}\Vert rz_0\Vert^n\nonumber\hspace{3cm}\\[1mm]
&&+\sum_{n=1}^\infty \frac{1}{n+1}\left(\sum_{k=1}^n\frac{\Gamma(k-1+\beta)}{\Gamma(k)\Gamma(\beta)}\Vert rz_0\Vert^{k-1}  \Vert P_{n-k+1}(rz_0)\Vert_\infty\right)\nonumber\\[1mm]
&=&\lambda \sum_{n=0}^\infty \frac{1}{n+1}\frac{\Gamma(n+\beta) }{\Gamma(n+1)\Gamma(\beta)}r^n+\sum_{n=1}^\infty \frac{1}{n+1}\left(\sum_{k=1}^n\frac{\Gamma(k-1+\beta)}{\Gamma(k)\Gamma(\beta)}r^n \lambda^{n-k}(1-\lambda^2)\right)\nonumber\\
&=&\lambda \sum_{n=0}^\infty \frac{1}{n+1}\frac{\Gamma(n+\beta) }{\Gamma(n+1)\Gamma(\beta)}r^n+(1-\lambda^2)\sum_{n=1}^\infty \frac{r^n}{n+1}\left(\sum_{k=1}^n\frac{\Gamma(k-1+\beta)}{\Gamma(k)\Gamma(\beta)}\lambda^{n-k}\right)\nonumber\\
&=&\frac{\lambda}{r}\int_{0}^r\frac{1}{(1-t)^{\beta}} d t+(1-\lambda^2)\sum_{n=1}^\infty\frac{r^n}{n+1}\left(\sum_{k=1}^n\frac{\Gamma(n-k+\beta)}{\Gamma(n-k+1)\Gamma(\beta)}\lambda^{k-1}\right)\nonumber\eea
\bea&=&\frac{\lambda}{r}\left(\frac{1-(1-r)^{1-\beta}}{1-\beta}\right)+(1-\lambda^2)\sum_{n=1}^\infty\frac{r^n}{n+1}\left(\sum_{k=1}^n\frac{\Gamma(n-k+\beta)}{\Gamma(n-k+1)\Gamma(\beta)}\lambda^{k-1}\right)\nonumber\\[1mm]
&=&\frac{1}{r}\left(\frac{1-(1-r)^{1-\beta}}{1-\beta}\right)-\frac{1-\lambda}{r}\left(\frac{1-(1-r)^{1-\beta}}{1-\beta}\right)\nonumber\\[1mm]
&&+(1-\lambda^2)\sum_{n=1}^\infty\frac{r^n}{n+1}\left(\sum_{k=1}^n\frac{\Gamma(n-k+\beta)}{\Gamma(n-k+1)\Gamma(\beta)}\lambda^{k-1}\right)\nonumber\\[1mm]
\label{g3}&=&\frac{1}{r}\left(\frac{1-(1-r)^{1-\beta}}{1-\beta}\right)-\frac{1-\lambda}{r}\left(\frac{3\left(1-(1-r)^{1-\beta}\right)}{1-\beta}-\frac{2\left((1-r)^{-\beta}-1\right)}{\beta}\right)\nonumber\\&&+G_7(\lambda, r),\eea
where 
\beas G_7(\lambda, r)&=&\frac{1-\lambda}{r}\left(\frac{2\left(1-(1-r)^{1-\beta}\right)}{1-\beta}-\frac{2\left((1-r)^{-\beta}-1\right)}{\beta}\right)\hspace{4cm}\\[1mm]
&&+(1-\lambda^2)\sum_{n=1}^\infty\frac{r^n}{n+1}\left(\sum_{k=1}^n\frac{\Gamma(n-k+\beta)}{\Gamma(n-k+1)\Gamma(\beta)}\lambda^{k-1}\right)\\
&=&2(1-\lambda)\left(\frac{1}{r}\int_{0}^r\frac{1}{(1-t)^{\beta}} d t-\frac{1}{r}\int_{0}^r\frac{1}{(1-t)^{\beta+1}} d t\right)\\[1mm]
&&+(1-\lambda^2)\sum_{n=1}^\infty\frac{r^n}{n+1}\left(\sum_{k=1}^n\frac{\Gamma(n-k+\beta)}{\Gamma(n-k+1)\Gamma(\beta)}\lambda^{k-1}\right)\\[1mm]
&=&2(1-\lambda)\left(\sum_{n=0}^\infty \frac{1}{n+1}\frac{\Gamma(n+\beta) }{\Gamma(n+1)\Gamma(\beta)}r^n\right.\\&&\left.-\sum_{n=0}^\infty \frac{1}{n+1}\left(\sum_{k=0}^n\frac{\Gamma(n-k+\beta)}{\Gamma(n-k+1)\Gamma(\beta)}\right) r^n\right)\\[1mm]
&&+\frac{(1-\lambda^2)}{\lambda}\sum_{n=0}^\infty\frac{r^n}{n+1}\left(\sum_{k=0}^n\frac{\Gamma(n-k+\beta)}{\Gamma(n-k+1)\Gamma(\beta)}\lambda^{k}-\frac{\Gamma(n+\beta)}{\Gamma(n+1)\Gamma(\beta)}\right)\\[1mm]
&=&\sum_{n=0}^\infty \frac{1}{n+1}\left(-\frac{(1-\lambda)^2}{\lambda}\frac{\Gamma(n+\beta)}{\Gamma(n+1)\Gamma(\beta)}-2(1-\lambda)\sum_{k=0}^n\frac{\Gamma(n-k+\beta)}{\Gamma(n-k+1)\Gamma(\beta)}\right.\\[1mm]
&&\left.+\frac{(1-\lambda^2)}{\lambda}\sum_{k=0}^n\frac{\Gamma(n-k+\beta)}{\Gamma(n-k+1)\Gamma(\beta)}\lambda^{k}\right)r^n.\eeas
Using the identity (\ref{ceq3}) and letting $\lambda\to 1^-$, we have 
\bea
\label{g4}G_7(\lambda, r)= O\left((1-\lambda)^2\right).\eea
Furthermore, it is evident that for $r>R_\beta$, the following inequality holds
\bea\label{g5} \frac{3\left(1-(1-r)^{1-\beta}\right)}{1-\beta}-\frac{2\left((1-r)^{-\beta}-1\right)}{\beta}<0.\eea
From (\ref{g3}), (\ref{g4}) and (\ref{g5}), we conclude that $R_\beta$ cannot be improved. This completes the proof.
\end{proof}
\section*{ Declarations}
\noindent{\bf Authors' contributions:} All authors have equal contribution to complete the manuscript. All of them read and approved the final manuscript.\\
{\bf Conflict of Interest:} Authors declare that they have no conflict of interest.\\
{\bf Availability of data and materials:} Not applicable.\\
{\bf Acknowledgment:} Second author is supported by University Grants Commission (IN) fellowship (No. F. 44 - 1/2018 (SA - III)).

\end{document}